\documentclass[a4paper,12pt,reqno]{article}

\usepackage{amsmath}
\usepackage[hmargin=1in,vmargin=1in,nohead]{geometry}
\usepackage{amssymb}
\usepackage{amsfonts}
\usepackage{amsthm}
\usepackage[french,english]{babel}
\usepackage{enumerate}
\usepackage{graphicx}
\usepackage{a4wide}
\usepackage[latin1]{inputenc}  
\usepackage[T1]{fontenc}

\RequirePackage[OT1]{fontenc}

\theoremstyle{plain}
\newtheorem{theorem}{Theorem}

\newtheorem{corollary}{Corollary}

\newtheorem{definition}{Definition}

\newtheorem{lemma}{Lemma}
\newtheorem{proposition}{Proposition}

\numberwithin{equation}{section}

\newcommand{\Real}{\mathbb R}
\newcommand{\N}{\mathbb N}
\newcommand{\Esp}{\mathbb E}
\newcommand{\Prob}{\mathbb P}
\newcommand{\eps}{\varepsilon}

\newcommand{\Inv}{\frac{1}}

\newcommand{\n}{\mathbf{n}}

\newcommand{\E}{\mathcal E}
\newcommand{\h}{h}
\newcommand{\un}{\mathbf 1}
\newcommand{\Pl}{\Prob_\lambda}
\newcommand{\El}{\Esp_\lambda}

\date{}
\title{ \bf Excursions of the integral of the Brownian motion}
\author{ \bf Emmanuel Jacob \\
\\
\emph {Laboratoire de Probabilit\'es et Mod\`eles Al\'eatoires} \\
\emph {Universit\'e Pierre et Marie Curie, Paris VI}}

\begin{document}

\maketitle{}

\begin{abstract}
The integrated Brownian motion is sometimes known as the Langevin
process. Lachal studied several excursion laws induced by the
latter. Here we follow a different point of view developed by Pitman
for general stationary processes. We first construct a stationary
Langevin process and then determine explicitly its stationary
excursion measure. This is then used to provide new descriptions of
It\^o's excursion measure of the Langevin process reflected at a
completely inelastic boundary, which has been introduced recently by
Bertoin.
\end{abstract}

\selectlanguage{francais}
\begin{abstract}
L'intégrale du mouvement Brownien est parfois appelée processus de
Langevin. Lachal a étudié plusieurs lois d'excursions qui lui sont
associées. Nous suivons ici un point de vue différent, développé par
Pitman, pour les processus stationnaires. Nous construisons d'abord
un processus de Langevin stationnaire avant d'en déterminer
explicitement la mesure d'excursion stationnaire. Ce travail permet
alors de fournir une nouvelle description de la mesure d'excursion
d'It\^o du processus de Langevin réfléchi sur une barrière
inélastique, introduit récemment par Bertoin.
\end{abstract}

\selectlanguage{english}

 \vspace{9pt} \noindent {\bf Key words.}
{Langevin process}, {stationary process}, {excursion measure},
{time-reversal}, \\{$h$-transform.}
\par \vspace{9pt} \noindent {\bf e-mail.} {emmanuel.jacob@normalesup.org}
\par

\section{Introduction}

The Langevin process in a non-viscous fluid is simply defined as the integrated Brownian motion, that is:
$$ Y_t= Y_0 + \int_0^t W_s ds,$$
where $W$ is a Brownian motion started an arbitrary $v \in \Real$
(so $v$ is the initial velocity of $Y$). The Langevin process is not
Markovian, but the pair $Z=(Y,W)$, which is sometimes known as the
Kolmogorov process, enjoys the Markovian property. We refer to
Lachal \cite{Lachal03} for a rich source of information on this
subject.

Lachal \cite{Lachal03} has studied in depth both the ``vertical" and
``horizontal" excursions of the Brownian integral. The purpose of
this work is to follow a different (though clearly related) point of
view, which has been developed in a very general setting by Pitman
\cite{Pitman}. Specifically, we start from the basic observation
that the Lebesgue measure on $\Real^2$ is invariant for the
Kolmogorov process, so one can work with a stationary version of the
latter. The set of times at which the stationary Kolmogorov process
visits $\{0\}\times \Real$ forms a random homogeneous set in the
sense of Pitman, and we are interested in the excursion measure
$Q_{ex}$ that arises naturally in this setting. We shall show that
$Q_{ex}$ has a remarkably simple description and fulfills a useful
invariance property under time-reversal. We then study the law of
the excursions of the Langevin process away from $0$ conditionally
on its initial and final velocity, in the framework of Doob's
$\h$-transform. Finally, we apply our results to investigate the
Langevin process reflected at a completely inelastic boundary, an
intriguing process which has been studied recently by Bertoin
\cite{reflecting, SDE}. In particular we obtain new expressions for
the It\^o measure of its excursions away from $0$.

 \section{Preliminaries}

 In this section we introduce some general or intuitive notations and
recall some known results that we will use later on. We write $Y$
for the Langevin process, $W$ for its derivative, and $Z$ for the
Kolmogorov process $(Y,W)$, which, unlike $Y$, is Markovian.

The law of the Kolmogorov process with initial condition $(x,u)$
will be written $\Prob_{x,u}^+$, and the expectation under this
measure $\Esp_{x,u}^+$. Here, the exponent $+$ refers to the fact
that the time parameter $t$ is nonnegative. We denote by
$p_t(x,u;dy,dv)$ the probability transitions of $Z$, and by
$p_t(x,u;y,v)$ their density. For $x,u,y,v \in \Real$, we have:

$$p_t(x,u;y,v) dudv:= p_t(x,u;dy,dv) := \Prob_{x,u}^+ (Z_t \in dydv).$$

These densities are known explicitly and given by:
\begin{equation}
p_t(x,u;y,v) = \frac{\sqrt 3}{\pi t^2} \exp \Big[ -\frac{6}{t^3}
(y-x-tu)^2 + \frac{6}{t^2} (y-x-tu)(v-u)-\frac{2}{t} (v-u)^2 \Big].
\end{equation}

One can check from the formula that the following identities are satisfied:
\begin{eqnarray}
p_t(x,u;y,v) &=& p_t(0,0;y-x-ut,v-u), \\
p_t(x,u;y,v) &=& p_t(-x,-u;-y,-v), \\
p_t(x,u;y,v) &=& p_t(x,v;y,u).
\end{eqnarray}
A combination of these formulas gives

\begin{eqnarray}
\label{duality} p_t(x,u;y,v) &=& p_t(y,-v;x,-u),
\end{eqnarray}
that we will use later on. See for example the equations (1.1), p
122, and (2.3), p 128, in \cite{Lachal03}, for references.

The semigroup of the Kolmogorov process will be written $P_t$. If
$f$ is a nonnegative measurable function, we have:

 $$P_t f(x,u):= \Esp_{x,u}^+\big(f(Y_t,W_t)\big)= \int_{\Real^2} dydv p_t(x,u;y,v) f(y,v).$$

The law of the Kolmogorov process with initial distribution given by
the Lebesgue measure $\lambda$ on $\Real^2$ will be written
$\Prob_\lambda^+$. It is given by the expression:
$$ \Pl^+ = \int_{\Real^2} \lambda(dx,du) \Prob_{x,u}^+. $$

 Although $\lambda$ is only a $\sigma$-finite measure, the expression
above still defines what we call a stochastic process in a
generalized sense (this is a common generalization, though). We
still use all the usual vocabulary, such as the law of the process,
the law of the process at the instant $t$, even though this laws are
now $\sigma$-finite measures and not probabilities.

 Finally, we recall the scaling property of the Langevin process:

 \begin{eqnarray}\label{scaling}
\Esp_{x,u}^+ \Big(F \big((Y_t)_{t\ge 0}\big)\Big) = \Esp_{k^{3}x,
ku}^+ \Big(F\big( (k^{-3} Y_{k^2 t})_{t \ge 0}\big)\Big),
\end{eqnarray}
where $F$ is any nonnegative measurable functional.

\section{Stationary Kolmogorov process}

The stationary Kolmogorov process is certainly not something new for
the specialists, as it is known that $\lambda$ is an invariant
measure for the Kolmogorov process. This section still gives, for
the interested reader, a rigorous introduction to the stationary
Kolmogorov process, including a duality property that allows us to
consider the effect of time-reversal, which will be a central point
of this paper.

\subsection{Stationarity and duality lemmas}

We write $\lambda$ for the Lebesgue measure on $\Real^2$.

\begin{lemma}
For any nonnegative measurable functions $f,g$ on $\Real^2$ and
every $t\geq0$, we have:
$$ \El^+ \big(f(Y_t,W_t)\big) = \El^+ \big(f(Y_0,W_0)\big), $$
and:
$$ \El^+  \big(f(Y_0,W_0) g(Y_t, W_t)\big) = \El^+ \big(f(Y_t,-W_t) g(Y_0,-W_0)\big). $$
\end{lemma}

This lemma states the (weak) stationarity of the measure $\lambda$ and a duality property of the process under this measure.

\begin{proof}

Let $f$ be a nonnegative measurable function on $\Real^2$, and $t$
be a positive real number.

\begin{eqnarray*}
 \El^+ \big(f(Y_t,W_t)\big) &:=& \int dx du \Esp_{x,u}^+ \big(f(Y_t,W_t)\big) \\
  &=& \int dx du \int dy dv  \: p_t(x,u;y,v) f(y,v) \\
  & =& \int \int dx du dy dv \: p_t(y,-v;x,-u) f(y,v)\ \ \  \text{   by (\ref{duality})} \\
   &=& \int dy dv f(y,v) \int dx du \: p_t(y,-v;x,u) \\
   &=& \int dy dv f(y,v)\\
   &=& \El^+ \big(f(Y_0,W_0)\big),
\end{eqnarray*}
where in the fourth line we made the simple change of variables $u
\rightarrow -u$.

For the second part, let $f$ and $g$ be two nonnegative measurable
functions, and $t$ a positive real number.

\begin{eqnarray*}
  & & \El^+ \big(f(Y_0,W_0)g(Y_t,W_t)\big) \\ 
  &=& \int dx du f(x,u) \int dy dv  \: p_t(x,u;y,v) g(y,v) \\
  &=& \int \int dx du dy dv f(x,-u) g(y,-v) \: p_t(x,-u;y,-v) \\
   &=& \int dy dv g(y,-v) \int dx du f(x,-u) p_t(y,v;x,u) \ \ \ \ \ \  \text{   by (\ref{duality}) again} \\
   &=& \El^+ \big(g(Y_0,-W_0) f(Y_t,-W_t) \big).
\end{eqnarray*}
The lemma is proved.
\end{proof}
We immediately deduce the following corollary.

\begin{corollary}
 For any $t>0$, we have:

 1) Stationarity: The law of the process $(Y_{t+s}, W_{t+s} ) _{s \geq 0}$ under $\Pl^+$ is $\Pl^+$.

 2) Duality: the laws of the processes $(Y_{t-s}, -W_{t-s} ) _{0 \leq s \leq t}$ and $(Y_s, W_s ) _{0 \leq s \leq t}$ under $\Pl^+$ are the same.
\end{corollary}

This corollary  provides a probabilistic interpretation of the stationarity and the duality property, here stated in a strong sense. Strong sense means that we consider here the whole trajectory and not merely the two-dimensional time-marginals. We thus see that the stationarity is a property of invariance of the process by time-translation, and that the duality is a property of symmetry of the process by time-reversal.

\begin{proof}
 As the processes we consider are continuous, their laws are determined by their finite-dimensional marginals.
  The strong stationarity is a simple consequence from the weak stationarity and the
Markov property, while the strong duality needs a bit more work. Let
$n \in \N$, let $0=t_0 \leq t_1 \leq ... \leq t_n$ be real numbers
and let $f_0$, $f_1$, ..., $f_n$ be  $n+1$ nonnegative measurable
functions. We have to prove that the following equality is satisfied
(recall $Z =(Y,W)$):
\begin{eqnarray} \label{strongstationarity}
\El ^+ \big[ f_0(Z_0) f_1(Z_{t_1}) ... f_n(Z_{t_n})\big]
  &=& \El^+ \big[ f_n(Z_0) f_{n-1}(Z_{t_n-t_{n-1}}) ...
f_1(Z_{t_n-t_1}) f_0(Z_{t_n}) )\big]. \qquad
\end{eqnarray}

 This is checked by induction on $n$. For $n=1$, this is nothing else
than the weak duality. We suppose now that the identity
(\ref{strongstationarity}) is true for any integer strictly smaller
than $n$. We have:

\begin{eqnarray*}
 & & \El ^+\big[ f_0(Z_0) f_1(Z_{t_1}) ... f_n(Z_{t_n})\big] \\
 &=& \El^+\big[ f_0(Z_0) ... f_{n-1} (Z_{t_{n-1}}) \Esp _{Z_{t_{n-1}}}^+ [ f_n (Z_{t_n-t_{n-1}}) ] \big] \\
 &=& \El ^+ \big[  \Esp _{Z_0} ^+ [ f_n (Z_{t_n-t_{n-1}})] f_{n-1} (Z_0) f_{n-2} (Z_{t_{n-1}-t_{n-2}}) ... f_0 (Z_{t_{n-1}}) \big] \\
 &=& \El^+ \Big[ f_n (Z_{t_n-t_{n-1}}) \Esp_{Z_0} ^+ \big[ f_{n-1} (Z_0) f_{n-2} (Z_{t_{n-1}-t_{n-2}}) ... f_0 (Z_{t_{n-1}}) \big] \Big] \\
 &=& \El ^+ \Big[ f_n (Z_0) \Esp_{Z_{t_n-t_{n-1}}} ^+ \big[ f_{n-1} (Z_0) f_{n-2} (Z_{t_{n-1}-t_{n-2}}) ... f_0 (Z_{t_{n-1}}) \big] \Big] \\
 &=& \El ^+ \big[ f_n (Z_0) f_{n-1} (Z_{t_{n-1}}) ... f_0 (Z_{t_n})\big].
\end{eqnarray*}
To get the second equality, we used (\ref{strongstationarity}) with the functions $f_0$, ..., $f_{n-2}$ and $\tilde{f}_{n-1}: (x,u) \rightarrow
f_{n-1} (x,u) \Esp _{x,u}^+ \big[ f_n (Z_{t_n-t_{n-1}}) \big]$.
To get the fourth equality, we use the weak duality with times 0 and $t_n-t_{n-1}$.

This completes our proof.
\end{proof}

\subsection{Construction of the stationary Kolmogorov process}

We are ready to construct the stationary Kolmogorov process with time parameter $t \in \Real$. First, we construct a process indexed by $\Real$ with a position $(x,u)$ at time 0. The process $(Z_t)_{t\in \Real}=(Y_t,W_t)_{t\in \Real}$ is such that $(Y_t,W_t)_{t\in \Real_+}$ has the law $\Prob_{x,u}^+$ and $ (Y_{-t},-W_{-t})_{t\in \Real_+}$ is an independent process and of law $\Prob_{x,-u}^+ $.
 The law of the process $(Z_t)_{t\in \Real}$ will be denoted by $\Prob_{x,u}$.

 \begin{definition}
 The stationary Kolmogorov process is the generalized process of law $\Pl$ given by:
 \begin{eqnarray}
 \Pl= \int dxdu \Prob_{x,u}.
 \end{eqnarray}
 \end{definition}

Lemma 1 and Corollary 1 still hold if we drop the superscript +.
 We stress that the stationary Kolmogorov process has a natural
filtration given by $\mathcal{F}_t = \sigma (\{Z_s\}_{-\infty < s
\leq t})= \sigma (\{Y_s\}_{-\infty < s \leq t})$. If $(Z_t)_{t\in
\Real}=(Y_t,W_t)_{t\in \Real}$, we call conjugate of $Z$ and write
$\overline Z$ for the process $\big(\overline Z_t\big)_{t\in
\Real}=(Y_t,-W_t)_{t\in \Real}$.
\begin{lemma}
 The stationary Kolmogorov process has the following properties:

\begin{enumerate}
  \item Under $\Pl$, The processes $Z$ and $\big(\overline Z_{-t}\big)_{ t\in \Real}$ have the same law. That is, the law $\Pl$ is invariant by time-reversal and conjugation.
  \item Under $\Pl$, the processes $(Y_t,W_t)_{t\in \Real}$ and $(Y_{t_0+t},W_{t_0+t})_{t\in \Real}$ have the same law for any $t_0 \in \Real$.
That is, the law $\Pl$ is invariant by time-translation.
  \item The process Z is a stationary Markov process under $\Pl$.
\end{enumerate}
\end{lemma}

\begin{proof}
 (1) Let us consider $Z$ a process of law $\Prob_{x,u}$. It is
immediate from the definition that the conjugate of the
time-reversed process, that is $\big(\overline Z_{-t}\big)_{t \in
\Real}$, is a process of law  $\Prob_{x,-u}$. The result follows.

 (2) Let us write $\Pl^{t_0}$ for the law of the process
$(Y_{t_0+t},W_{t_0+t})_{t\in \Real}$ under $\Pl$, and let us suppose
in this proof that $t_0$ is positive. We want to prove that
$\Pl^{t_0}$ and $\Pl$ are equal. It is enough to prove that for any
suitable functional $f$, $g$ and $h$, the expectations of the
variable $$f\big((Y_t)_{t \leq -t_0}\big) \ g\big((Y_t)_{-t_0 \leq t
\leq 0}\big) \ h\big((Y_t)_{0 \leq t}\big)$$ under these two
measures are equal \footnote{We take only functionals of Y and not
of W. This is in order to make the notations simpler and has no
incidence, as W can be recovered from Y by taking derivatives.}. On
the one hand, we have:

\begin{eqnarray*}
&& \El^{t_0} \big[f((Y_t)_{t \leq -t_0}) g((Y_t)_{-t_0 \leq t \leq 0}) h((Y_t)_{0 \leq t})\big] \\
&=& \El \big[f((Y_t)_{t \leq 0}) g((Y_t)_{0 \leq t \leq t_0}) h((Y_t)_{t_0 \leq t})\big] \\
&=& \El \Big[ \Esp_{Y_0,W_0} \big[f((Y_t)_{t \leq 0}) \big] \: \Esp_{Y_0,W_0} \big[ g((Y_t)_{0 \leq t \leq t_0})
\Esp_{Y_{t_0}, W_{t_0}} [h((Y_t)_{t_0 \leq t}) ] \big] \Big] \\
&=& \El \Big[\Esp_{Y_0,W_0} \big[f((Y_t)_{t \leq 0})\big] \: g\big((Y_t)_{0 \leq t \leq t_0}\big) \:
\Esp_{Y_{t_0}, W_{t_0}} \big[h((Y_t)_{t_0 \leq t}) \big]\Big] \\
&=& \El \big[F(Y_0,W_0) \: g((Y_t)_{0 \leq t \leq t_0}) \:
H(Y_{t_0}, W_{t_0})\big],
\end{eqnarray*}
where we wrote $F(x,u)=\Esp_{x,u}\big[f((Y_t)_{t \leq 0})\big]$ and
$H(x,u)= \Esp_{x,u} \big[h((Y_t)_{t_0 \leq t}) \big]$. To get the
third line we use the independence of $(Y_t)_{t \leq 0}$ and
$(Y_t)_{t \geq 0}$ conditionally on $(Y_0, W_0)$ and the Markov
property of $(Y_t)_{t \geq 0}$ at time $t_0$.

On the other hand, we have:

\begin{eqnarray*}
&& \El \big[f((Y_t)_{t \leq -t_0}) g((Y_t)_{-t_0 \leq t \leq 0}) h((Y_t)_{0 \leq t})\big] \\
&=& \El \big[\Esp_{Y_{-t_0},W_{-t_0}}[f((Y_t)_{t \leq 0})] \: g((Y_t)_{-t_0 \leq t \leq 0})\: \Esp_{Y_0, W_0} [h((Y_t)_{0 \leq t}) ]\big] \\
 &=&\El \big[F(Y_{-t_0}, W_{-t_0}) \: g((Y_t)_{-t_0 \leq t \leq 0}) \: H(Y_0,W_0)\big]\\
  &=&\El \big[H(Y_0,W_0) \: g((Y_{t_0-t})_{0 \leq t \leq t_0}) \: F(Y_{t_0}, W_{t_0})\big],
\end{eqnarray*}
where $F$ and $H$ are defined above and we used the time-reversal invariance property for $\Pl$ to get the last line.

Now, the fact that the two expressions we get are equal is a direct consequence of the duality property stated in a strong sense.

 (3) In this third statement the important word is the word
\emph{Markov}, not the word \emph{stationary}. Indeed the Markov
property for negative times is not immediate in the definition of
$\Pl$. But the Markov property for positive times is, and this
combined with the stationarity immediately gives the Markov property
for any time.
\end{proof}

 In the following, we will speak about the stationary Kolmogorov
process for the process $(Y,W)$ under $\Pl$, and about the
stationary Langevin process for the process $Y$ under $\Pl$.

 Before speaking about excursions of these processes, let us notice
that we could have constructed the stationary Kolmogorov process
starting from time $-\infty$ with using just the stationarity (and
not the duality). The way to do it is to consider the family of
measures $(^t\Pl^+)_{t \leq 0}$, where $^t\Pl^+$ is the measure of
the Kolmogorov process starting from the measure $\lambda$ at time
$t$ . The stationarity gives us that these measures are compatible.
We thus can use Kolmogorov extension theorem and construct the
measure starting from time $-\infty$.

 In this construction, though, the nontrivial fact is that the
process is invariant by time-reversal, and we need the duality
property to prove it.

\section{Excursions of the stationary Langevin process}

Until now we considered the Langevin -- or the Kolmogorov -- process
on an infinite time interval. In this section we will deal with the
same process killed at certain hitting times. For the sake of
convenience, we use here the notation $Y$ for the canonical smooth
process and $W$ for its derivative.

\subsection{Stationary excursion measure}

We will now study the \emph{stationary excursion measure} for a stationary process given by Pitman in~\cite{Pitman}.

 If $t$ is a time such that $Y_t=0$ and $W_t \neq 0$, we will write
$e^t$ or $(e^t_s)_{0 \leq s \leq \zeta}$ for the excursion of $Y$
away from 0 started at time $t$, and $\zeta$ for its lifetime, that
is, $\zeta(e^t):= \inf \{s>0: Y_{t+s}=0\}$ and $e^t_s:=Y_{t+s}$ for
${0 \leq s \leq \zeta}.$

 It belongs to the \emph{set of vertical excursions} $\E_0$, that is,
the set of continuous functions $t \rightarrow Y_t$, defined on
$\Real_+$, that have a c\`adl\`ag right-derivative $W$, such that
$Y$ starts from zero ($Y_0=0$), $Y$ leaves immediately zero ($Y$ has
a strictly positive lifetime $\zeta(Y)$), and dies after its first
return to $0$. This definition is inspired by the terminology of
Lachal \cite{Lachal03}, except that he considers the set of vertical
excursions for the two-dimensional process.

 We write $\Prob_{x,u}^\partial$ for the law of the Langevin process
starting with position $x$ and velocity $u \ne 0$, and killed at its
first return-time to $0$. So it is a law on the set of vertical
excursions, and under $\Prob_{0,u}^+$, the excursion starting at
time 0 is written $e^0$ and has law $\Prob_{0,u}^\partial$.

 Considering the stationary Langevin process and the homogeneous set
$\{t, Y_t = 0\}$, we define in the sense of Pitman \cite{Pitman} the
stationary excursion measure:

\begin{definition}
We call \emph{stationary excursion measure} of the stationary Langevin process, and we write $Q_{ex}$, the measure given by:
\begin{eqnarray}
\label{Pit excursion} Q_{ex} ( \bullet) = \Esp_\lambda\big[ \verb!#! \{0<t<1, Y_t=0, e^t \in \bullet \}\big].
\end{eqnarray}
\end{definition}

We stress that this measure does not give a finite mass to the set
of excursions with lifetime greater than 1, contrarily to the It\^o
excursion measure of a Markov process. By a slight abuse of
notation, when $A$ is an event, we will write $Q_{ex} (\un_A)$ for
$Q_{ex} (A)$.

We stress that for convenience we focus here and thereafter on the
Langevin process; clearly this induces no loss of generality as the
Kolmogorov process can be recovered from the Langevin process by
taking derivatives. For instance, the law of the two-dimensional process
 $(Y,W)$ under $Q_{ex}$ is equal to the stationary excursion measure for the stationary Kolmogorov process and the homogeneous set $\{t, (Y_t,W_t) \in \{0\}\times \Real\}$.

Our main result is the following:

\begin{theorem}
1) There is the identity:
\begin{equation} \label{Q_ex}
Q_{ex} (Y \in de)= \int_{u=-\infty}^{+\infty}{|u|
\Prob_{0,u}^{\partial}\big(Y \in de\big) du }.
\end{equation}

2) The measure $Q_{ex}$ is invariant by time-reversal (at the
lifetime): Namely, the measure of $Y$ under $Q_{ex}$ is the same as
that of $\widehat{Y}$ under $Q_{ex}$, where $\widehat{Y}$ is defined
by
$$ \widehat{Y}_s = Y_{\zeta - s} \hbox{ for } 0 \leq s \leq \zeta. $$
\end{theorem}
Let us adopt the notation $\widehat Q_{ex}$ for the law of $\widehat
Y$ under $Q_{ex}.$ The second part of the theorem can be written
$\widehat Q _{ex}=Q_{ex}.$

 Let a Langevin process start from location $0$ and have initial
velocity distributed according to $|u|du$. Then the distribution of
its velocity at the first instant when it returns to $0$ is again
$|u|du$.

 This remarkable fact can be proved directly as follows. We use the
formula found by McKean~\cite{McKean}, which gives, under
$\Prob_{0,u}$, the joint density of $\zeta$ and $W_{\zeta}$, and
which specifies the density of $W_{\zeta}$. For $u>0$ and $v \geq0$,
we have:

\begin{eqnarray} \label{mcKean joint}
 \Prob_{0,u} (\zeta \in ds, -W_{\zeta^-} \in dv)= ds dv
 \frac{3v}{\pi \sqrt 2 s^2}
\exp\Big(-2\frac{v^2-uv+u^2}{s}\Big) \int_0^\frac{4 uv}{s}
{\rm e}^{-\frac{3\theta}{2}} \frac{d \theta}{\sqrt{\pi\theta}},
\end{eqnarray}
and in particular:
\begin{equation}
\label{mcKean} \Prob_{0,u}(-W_{\zeta^-} \in dv) =
\frac{3}{2\pi}\frac{u^\Inv{2}v^\frac{3}{2}}{u^3+v^3} dv.
\end{equation}

This formulas naturally still hold when you replace $\Prob_{0,u}$ by
$\Prob_{0,u}^\partial$ and $W_{\zeta}$ by $W_{\zeta^-}$.

In the calculation, we actually just need the second formula. Let
$v$ be any positive real number. We have:

\begin{eqnarray*}
Q_{ex} (W_{\zeta^-} \in dv) &=& \int_{u=-\infty}^{0}{|u| \Prob_{0,u}^{\partial}(W_{\zeta^-} \in dv) du } \\
&=& \int_{u=0}^{+\infty} |u| \Prob_{0,u}(-W_{\zeta^-} \in dv) du \\
&=& v dv \int_{u=0}^{+\infty} \frac{3}{2\pi}\frac{u^\frac{3}{2}v^\Inv{2}}{u^3+v^3}  du \\
&=& v dv.
\end{eqnarray*}

 The integral gives one as it is the integral of the density of
$-W_{\zeta}$ under $\Prob_{0,v}$, thanks to (\ref{mcKean}). The case
$v$ negative is similar and gives us $Q_{ex} (W_{\zeta^-} \in dv)=
-v dv$, as claimed.

 \begin{proof}[Proof of Theorem 1]

1) This proof is mainly a combination of
 the work of Pitman~\cite{Pitman} translated to the Langevin process, and of known
  results on the Langevin process, results that we can find in~\cite{Lachal03}.

We recall and adapt some of their notations.

 In \cite{Lachal03}, we consider the Langevin process on positive
times, and the last instant that the process crosses zero before a
fixed time $T$ is written $\tau_T^-$. In \cite{Pitman}, we write
$G_u$ for the last instant before $u$ that the stationary process
crosses zero. The variable $G_u$ can take finite strictly negative
values, while the variable $\tau_T^-$  cannot. If $T$ is a positive
time, then we can write $\tau_T^-=\un_{G_T \geq 0} G_T $.

 In \cite{Pitman}, the part (iv) of the Theorem (p 291), rewritten
with our notations, states \footnote{Actually, the article of Pitman
states $\Pl(-\infty< G_u< u, e^{G_u} \in de)=Q_{ex}(de) \zeta(e)$
for any $u \in \Real$.} :

\begin{eqnarray} \label{Pitman's}
\Pl(-\infty< G_0< 0, e^{G_0} \in de)=Q_{ex}(de) \zeta(e)
\end{eqnarray}

In \cite{Lachal03}, the Lemma 2.5, p 129, states an important and simple relation, that can be written
\begin{equation}\label{vPv->P(zeta)}
    \Prob_{0,v}^\partial\big((Y_t,W_t) \in dx du\big) |v| dv dt = \Prob_{x,-u}(\zeta \in dt, -W_{\zeta} \in dv) dx du,
\end{equation}
 and that is a main tool used to prove the Theorem 2.6. The points 1) and 4) of this Theorem state:

 \begin{eqnarray}
\label{Lachal's1}\Prob_{x,u}^+\big\{(\tau_T^-, W_{\tau_T^-}) \in ds dv\big\} / ds dv = |v| p_s(x,u,0,v) \Prob_{0,v}^+\{\zeta > T-s\},\ \ \ \  \\
\label{Lachal's2}\Esp_{x,u}^+ \big[F(\tau_T^-,e^{\tau_T^-}_Z)|(\tau_T^-,W_{\tau_T^-}) = (s,v)\big] = \Esp_{0,v}^+[F(s, e^0_Z)|\zeta>T-s],\ \ \ \
\end{eqnarray}
where $F$ is any suitable functional, and $e_Z^t$ denotes the
excursion of the two-dimensional process started at a time $t$ such
that $Y_t=0$.

Let us now begin. From (\ref{Pitman's}), it is sufficient to prove the following:
\begin{eqnarray} \label{toprove}
\Pl(-\infty<G_0 <0, e^{G_0} \in de)= \zeta(e)
\int_{u=-\infty}^{+\infty}{|u| \Prob_{0,u}^{\partial}(Y \in de) du
}. \ \ \ \
\end{eqnarray}

We start from:
\begin{eqnarray*}
\Pl(-\infty<G_0<0, e^{G_0} \in de) &=& \lim_{T \rightarrow \infty} \Pl(-T<G_0<0, e^{G_0} \in de), \\
&=& \lim_{T \rightarrow \infty} \Pl(0<G_T<T,  e^{G_T}\in de), \\
&=& \lim_{T \rightarrow \infty} \int dxdu \Prob_{x,u}(0<G_T<T,
e^{G_T} \in de).
\end{eqnarray*}
Hence we have:

\begin{equation*}
\Pl(-\infty<G_0<0, e^{G_0} \in de) = \lim_{T \rightarrow \infty} \int dxdu \Prob_{x,u}(0<\tau_T^-<T, e^{\tau_T^-} \in de).
\end{equation*}

Let us write the term in the limit.
\begin{eqnarray*}
&&\int dxdu \ \ \Prob_{x,u}(0<\tau_T^-<T, e^{\tau_T^-} \in de) \\
 &=& \int dxdu \int \Prob_{x,u}\big((\tau_T^-, W_{\tau_T^-}) \in ds dv\big)
\ \ \ \  \Prob_{x,u}\big( e^{\tau_T^-}\in de | (\tau_T^-, W_{\tau_T^-})= (s,v)\big), \\
 &=& \int dxdu \int \Prob_{x,u}^+\big((\tau_T^-, W_{\tau_T^-}) \in ds dv\big)
\ \ \ \  \Prob_{x,u}^+\big( e^{\tau_T^-}\in de | (\tau_T^-, W_{\tau_T^-})= (s,v)\big), \\
&=& \int dxdu \int dsdv |v| p_s(x,u,0,v) \Prob_{0,v}^+\{\zeta > T-s\}
\Prob_{0,v}^+ (e^0 \in de | \zeta >T-s),
\end{eqnarray*}
where the integrals cover $(x,u) \in \Real^2$,
$(s,v) \in [0,T] \times \Real$.
In the last line we used (\ref{Lachal's1}), and (\ref{Lachal's2}) with the simple function $F\big(s,(Y,W)\big)= \mathbf{1}_{Y \in de}.$

By Fubini, the last expression is also equal to

\begin{eqnarray*}
&& \int dv |v| \int ds \Big( \int dx du \  p_s(x,u,0,v) \Big) \Prob_{0,v}^+ (e^0 \in de, \zeta >T-s).\\
&=& \int dv |v| \int_0^T ds \Prob_{0,v}^\partial (Y \in de, s>T-\zeta(e)) \\
&=& \int dv |v| \Prob_{0,v}^\partial (Y \in de) (\zeta(e) \wedge T),
\end{eqnarray*}
where we get the second line because
$$\int dx du \  p_s(x,u,0,v) =\int dx du \  p_s(0,-v;x,-u)=1.$$

Now, letting $T$ go to $\infty$ gives us (\ref{toprove}) and completes our proof.

 2) We use the definition of $Q_{ex}$ by the equation (\ref{Pit
excursion}). The time-translation and time-reversal invariance of
$\El$ gives us the time-reversal invariance of $Q_{ex}$.
\end{proof}

We point out that the measure $Q_{ex}$ has a remarkably simple potential, given by:
\begin{equation}\label{Q_ex^infty}
    \int_{\Real_+} Q_{ex}\big((Y_t,W_t) \in \bullet\big) dt = \lambda (\bullet).
\end{equation}

\begin{proof}
This is a consequence of (\ref{Q_ex}) and (\ref{vPv->P(zeta)}), that gives:
\begin{eqnarray*}
\int_{\Real_+}Q_{ex}\big((Y_t,W_t) \in dx du\big) dt &=&
\int_{\Real_+} dt
\int_{-\infty}^{+\infty} |v| dv \Prob_{0,v}^\partial\big((Y_t,W_t) \in dx du\big)  \\
   &=& dx du \int_{\Real_+} \int_{-\infty}^{+\infty}  \Prob_{x,-u}(\zeta \in dt, -W_{\zeta} \in dv)  \\
   &=&  dx du.
\end{eqnarray*}
\end{proof}

 Finally, let us notice that we get a scaling property for the
stationary excursion measure, which is a simple consequence from
(\ref{scaling}) and (\ref{Q_ex}):
\begin{equation}
Q_{ex}\Big(F \big((Y_t)_{t\ge 0}\big)\Big) = k^{-2} Q_{ex}
\Big(F\big( (k^{-3} Y_{k^2 t})_{t \ge 0}\big)\Big),
\end{equation}
where $F$ is any nonnegative measurable functional.

\subsection{Conditioning and $h$-transform}

 In the preceding section we defined the stationary excursion
measure, we described it with a simple formula and we proved its
invariance by time-reversal. This is a global result for this
measure. Now we would like to provide a more specific description
according to the starting and ending velocities of the excursions.
That is,  we would like to define and investigate the excursion
measure conditioned to start with a velocity $u$ and end with a
velocity $-v$, that would be a probability measure written
$Q_{u;v}$.

 Let us first notice that the measure $Q_{ex}(W_0 \in du,
-W_{\zeta^-} \in dv)$ has support $\{(u,v) \in \Real^2, uv \ge 0\}$.
It has a density with respect to the Lebesgue measure, that we write
$\varphi(u,v)$. This density is given, for $u>0$, $v>0$ or $u<0$,
$v<0$, by:
\begin{eqnarray*}
\varphi(u,v) &=& \Inv{dudv} \Big( |u| \Prob_{0,u}^\partial(-W_{\zeta^-} \in dv) du \Big)\\
&=&\frac 3 {2 \pi} \frac {|u|^\frac 3 2 |v|^\frac 3 2} {|u|^3+ |v|^3}. 
\end{eqnarray*}

\begin{definition}
 We write $(Q_{u;v})_{uv>0}$ for a version of the conditional law of $Q_{ex}$ given the initial speed is $u$ and the final speed $-v$.
 That is, for $f:\Real^2\to \Real$ and $G:\E_0 \to \Real$ nonnegative measurable functionals, we have:
 \begin{equation} \label{version_loi_conditionelle}
 Q_{ex} \big(f(W_0,-W_{\zeta^-}) G\big) = \int Q_{u;v}(G) f(u,v) \varphi(u,v) du dv.
 \end{equation}
\end{definition}
It is clear that $Q_{-u;-v}$ is the image of $Q_{u;v}$ by the
symmetry  $Y \to -Y$, for almost all $(u, v)$, so that in the
following we will only be interested in $Q_{u;v}$ for $u>0$, $v>0$.

 From the time-reversal invariance of the stationary excursion
measure, i.e $\widehat Q_{ex}= Q_{ex}$, we deduce immediately the
following time-reversal property of the conditioned measures:
\begin{equation}\label{uv-vu}
    \widehat{Q}_{u;v} = Q_{v;u} \qquad \hbox{ for a.~a. } (u,v) \in (\Real_+)^2.
\end{equation}

 Recall from the formula (\ref{Q_ex}) that $|u|
\Prob_{0,u}^\partial$ is a version of the conditional law of
$Q_{ex}$ given the initial speed $u$. It follows that we have the
following formula:
\begin{equation}\label{Q_u}
    \Prob_{0,u}^\partial = |u|^{-1} \int Q_{u;v} \varphi(u,v) dv \quad \hbox{ for almost all } u>0
\end{equation}

 The measure $|u|^{-1} \varphi(u,v) dv$ is the law of $-W_{\zeta^-}$
under $\Prob_{0,u}^\partial$. Hence $Q_{u;v}$ is a version of the
conditional law of $\Prob_{0,u}^\partial$ given $-W_{\zeta^-}=-v$.
Before going on, we need precise informations on the variable
$-W_{\zeta^-}$ and its law, under different initial conditions. The
results we need are gathered in the following lemma. We take the
notations $\Real_+^*$ for $\Real_+ \backslash \{0\}$, and $D$ for
the domain $\big((\Real_+^*) \times \Real\big) \bigcup \big(\{0\}
\times (\Real_+^*) \big)$.

\begin{lemma} \label{lemme_technique}

$\bullet$ For any $(x,u)$  in $D$, the density of the law of the
variable $-W_{\zeta^-}$ under $\Prob_{x,u}^\partial$ with respect to
the Lebesgue measure on $(0,\infty)$ exists and is written
$\h_v(x,u)$ for $v>0$. We have:
\begin{eqnarray} \label{h_v}
\h_v(x,u) &=& v \Big[ \Phi_0 (x,u;-v) - \frac{3}{2\pi} \int_0^\infty \frac{\mu^\frac{3}{2}}{\mu^3+1}
\Phi_0(x,u;\mu v ) d\mu \Big], \ \ \ \ \ \ \ \
\end{eqnarray}
where $ \Phi_0(x,u;v):= \Phi(x,u;0,v)$ and
$$ \Phi(x,u;y,v) := \int_0^\infty p_t(x,u;y,v) dt.$$
For $x=0$, this formula can be simplified as:
\begin{eqnarray} \label{h_v0}
\h_v(0,u) &=& \frac{3}{2\pi}\frac{u^\Inv{2}v^\frac{3}{2}}{u^3+v^3}.
\end{eqnarray}

$\bullet$ The function  $(v,x,u) \rightarrow \h_v(x,u)$ is
continuous on $E:=\Real_+^* \times D$.  The function $\Phi_0$ is
continuous and differentiable on $D\times \Real$. Moreover, we have
the following equivalence for $v$ in the neighborhood of zero:
\begin{equation}
\label{equiv} \h_v(x,u) \sim \overline{\h_0}(x,u) v^\frac{3}{2},
\end{equation}
where $\overline{\h_0}(x,u)$ is given by
$$ \overline{\h_0}(x,u)=\frac{3}{\pi}\int \alpha^{-\Inv{2}} \frac {\partial \Phi_0}{\partial v}(x,u;\alpha) d\alpha.$$
For $x=0$, this formula can be simplified as
 $$\overline{\h_0}(0,u)= \frac{3u^\Inv{2}}{2 \pi}.$$
\end{lemma}
 This is a technical lemma, with a long proof that we report in the Appendix.

 The idea is now, thanks to this lemma, to prove that the law
$\Prob_{0,u}^\partial$  conditioned on the event $-W_{\zeta^-} \in
[v,v+\eta]$, has a limit when $\eta$ goes to zero. This limit is
necessarily $Q_{u;v}$ a.~s. Hence we get an expression for
$Q_{u;v}$, that will happen to be a bi-continuous version.

 Let us fix $u,v,t>0$, and let  $\phi_t$ be an
$\mathcal{F}_t$-measurable nonnegative functional. We have:

\begin{eqnarray*}
 \lim_{\eta \rightarrow 0} \Esp_{0,u}^\partial\big(\phi_t \un_{ \zeta >t} | -W_{\zeta^-} \in [v,v+\eta]\big)
&=& \lim_{\eta \rightarrow 0} \frac{ \Esp_{0,u}^\partial(\phi_t
\un_{ \zeta >t,-W_{\zeta^-} \in [v,v+\eta]} ) }
{\Prob_{0,u}^\partial(-W_{\zeta^-} \in [v,v+\eta]) } \\
&=& \Esp_{0,u}^\partial \left( \phi_t \un_{ \zeta >t} \lim_{\eta
\rightarrow 0} \frac{\Prob_{Y_t,W_t}^\partial (-W_{\zeta^-} \in
[v,v+\eta])} {\Prob_{0,u}^\partial(-W_{\zeta^-} \in [v,v+\eta])}
\right).
\end{eqnarray*}

 The limit exists and is equal to the quotient of $\h_v(Y_t,W_t)$ by $\h_v(0,u)$. Hence, we get:

\begin{eqnarray} \label{Q_u,v}
Q_{u;v} (\phi_t \un_{ \zeta >t}) = \Esp_{0,u}^\partial \left( \phi_t
\un_{ \zeta >t} \frac{\h_v(Y_t,W_t)}{\h_v(0,u)} \right).
\end{eqnarray}
for any $t>0$, any  $\mathcal{F}_t$-measurable functional $\phi_t$.

 From the continuity of $h$ we deduce that $Q_{u;v}$ is jointly continuous in $u, v$, $(u,v) \in (\Real_+^*)^2$.
Furthermore, thanks to (\ref{equiv}), when $v$ goes to zero, the
quotient goes to
$\frac{\overline{\h_0}(Y_t,W_t)}{\overline{\h_0}(0,u)}$. We deduce
that the measures $Q_{u;v}$ have a weak limit when $v$ goes to zero,
that we write $Q_{u;0}$. We have

\begin{eqnarray} \label{Q_u,0}
 Q_{u;0} (\phi_t \un_{ \zeta >t}) = \Esp_{0,u}^\partial \left( \phi_t \un_{ \zeta >t} \frac{\overline{\h_0}(Y_t,W_t)}{\overline{\h_0}(0,u)} \right).
 \end{eqnarray}

 This shows that these measures $Q_{u;v}$ make appear $\h$-transforms
of the usual probability transitions of the Langevin process
$\Esp_{0,u}$. The $h$-transforms are common when dealing with
conditioned Markov process, see for example \cite{Azema}, and in
particular the chapters \emph{4.7.} and \emph{6.4.} for the
connection with time-reversal.

 Informally, in the case of two processes in duality, changing the
initial condition for one process corresponds to changing the
probability transitions of the second process into an $h$-transform
of these probability transitions. The $h$-transform means the
measure ``conditioned" with using a certain harmonic function $h$,
that we can write explicitly.

 We finish this section with giving the scaling property of the
measures $Q_{u;v}$, that follows for example from (\ref{scaling})
and (\ref{version_loi_conditionelle}):

\begin{proposition} For any $u > 0$, $v \ge 0$, we have:
\begin{equation}\label{Q_ku,kv}
    Q_{u;v} \Big(F \big((Y_t)_{t\ge 0}\big)\Big) = Q_{ku;kv} \Big(F\big( (k^{-3} Y_{k^2 t})_{t \ge 0}\big)\Big),
\end{equation}
where F is any nonnegative measurable functional.
\end{proposition}

\section{Reflected Kolmogorov process}

We begin this section on a new basis, with introducing a process
that has been studied recently. This is only in a second part that
the definitions that we developed before will be used for that
process.

\subsection{Preliminaries on the reflected Kolmogorov process}

 The question of the existence of the Langevin process reflected at a
completely inelastic boundary was raised by B. Maury in 2004 in
\cite{Maury}. The answer came in \cite{reflecting}, where J. Bertoin
proves the existence of that process and its uniqueness in law. We
also mention another paper \cite{SDE} that studies the problem of
the reflected Langevin process from the point of view of stochastic
differential equations.

\begin{definition}
We say that $(X,V)$ is a Kolmogorov process reflected at a
completely inelastic boundary (or just reflected Kolmogorov process)
if it is a c\`adl\`ag strong Markov process with values in $\Real_+
\times \Real$ which starts from $(0,0)$, such that $V$ is the
right-derivative of $X$, and also:
$$\int_0^\infty \mathbf{1}_{ \{X_t=0\} } dt = 0 \hbox{     and     } \big( X_t=0 \ \Rightarrow \ V_t=0 \big) a.s., $$
and which ``evolves as a Kolmogorov process when $X>0$'', in the following sense:
\begin{quote}
    \emph{For every stopping time $S$ in the natural filtration $(\mathcal F_t)_{t\geq 0}$ of $X$,
    conditionally on  $X_S= x > 0$ and $V_S = v$,
  the shifted process $(X_{S+t})_{t \geq 0}$ stopped when hitting 0 is independent of $\mathcal F_S$, and has the distribution of a Langevin process
started with velocity $v$ from the location $x$ and stopped when
hitting 0.}
\end{quote}

We say that $X$ is a Langevin process reflected at a completely
inelastic boundary (or just reflected Langevin process) if $(X,V)$
is a reflected Kolmogorov process.
\end{definition}

 In the following we choose the vocabulary and the notations of the
one-dimensional process, that is the Langevin process, to state our
results.

 In his paper Bertoin gives an explicit construction of a reflected
Langevin process: Starting from a Langevin process $Y$, he first defines a process $\tilde{X}$ using Skorokhod's reflection:

$$\tilde{X}_t = Y_t- \inf_{0 \le s \le t} Y_s.$$
Let us notice that an excursion of that process does take off with
zero velocity. However, that process cannot be the right one because
$$ \int_0^\infty \un_{\{\tilde X _t=0\}} dt =\infty \  a.s,$$
while we require $$ \int_0^\infty \un_{\{\tilde X _t=0\}} dt =0 \
a.s.$$ Further, it is easy to check that $(\tilde X, \tilde V)$
fails to be Markovian. But Bertoin then introduces a change of time,
with writing $$ T_t := inf \Big \{ s \geq 0: \int_0^s \mathbf{1}_{
\tilde{X}_u >0 }du > t \Big \}$$ and $$X_t := \tilde{X} \circ T_t.
$$ This process $X$ is a reflected Langevin process. The same paper
also proves \footnote{The idea of the above construction is still a
central point of the proof} the uniqueness of the law of a reflected
Langevin process, so that we will speak about \emph{the} reflected
Langevin process. In the rest of the paper, we will concentrate our
attention on what is one of the first steps in the study of this
process, that is to say its It\^o excursion measure. We recall that
it is unique up to a multiplicative constant.

\subsection{Itô excursion measure of the reflected Langevin process}

In this section we will thus deal with the excursions of the
reflected Langevin process. For the sake of convenience, we use here
the notation $X$ for the canonical smooth process, $V$ for its
derivative.

We consider the ``set of ends of vertical excursions" $\E$, that is
the set of excursions, except that we do not require anymore that
the excursions should start from position 0. This set, endowed with
the supremum norm of the process and its derivative, is a metric
space including $\E_0$. In the following, we write $F: \E
\rightarrow \Real$ for a general continuous bounded functional which
is identically 0 on some neighborhood of the path $X \equiv 0$.

 We are ready to state a first formula, given\footnote{Actually
Bertoin states this result in a slightly different form, as the set
of excursions he considers is not the exactly same as the one we
consider here. Nevertheless, his argument still works in our
settings.} by Bertoin \cite{reflecting}:

\begin{proposition} \label{proposition} The following limit
$$\n\big(F(X)\big) := \lim_{x \rightarrow 0+} x^{-\Inv{6}} \Esp_{x,0}^{\partial}\big(F(X)\big),$$
exists and defines uniquely a measure on $\E$ with $\n({0})=0$, and
which support is included in $\E_0$. The measure $\n$ is an It\^o
excursion measure of the reflected Langevin process.
\end{proposition}

 This is to say, we get an expression for the It\^o excursion measure
of the reflected Langevin process as a limit of known measures.

 This result resembles the classical approximation of the It\^o
measure of the absolute value of the Brownian motion by $x^{-1}
\Prob_x^\partial$, where $\Prob_x^\partial$ is the law of the
Brownian motion starting from $x$ and killed when hitting 0.

 As a consequence of this expression, we can give the scaling
property of this measure, also mentioned in \cite{reflecting},
Proposition 2:

\begin{corollary}
We have:
$$\n  \Big(F \big((X_t)_{t\ge 0}\big)\Big) = k^\Inv 2 \n\Big(F\big( (k^{-3} X_{k^2 t})_{t \ge 0}\big)\Big),$$
for any nonnegative measurable functional $F$.
\end{corollary}

\begin{proof}
Let $F$ be a general continuous bounded
functional which is identically 0 on some neighborhood of the path
$e \equiv 0$. The proposition gives us:
\begin{eqnarray*}
  \n\big( F((X_t)_{t\ge0}) \big)&=& \lim_{x \rightarrow 0+} x^{-\Inv{6}} \Esp_{x,0}^{\partial}\big(F((X_t)_{t\ge 0})\big) \\
  &=& k^\Inv 2 \lim_{x \rightarrow 0+} (k^3 x)^{-\Inv{6}}
  \Esp_{k^3 x,0}^{\partial}\big(F((k^{-3}X_{k^2t})_{t\ge 0})\big) \qquad \hbox{ by (\ref{scaling})} \\
   &=& k^\Inv 2 \n\big(F((k^{-3}X_{k^2t})_{t\ge 0})\big).
\end{eqnarray*}
The result follows.
\end{proof}
 We give here two new expressions of the It\^o excursion measure of the reflected process. The first one is similar to the one
above, expressed as a limit. But it is a limit of laws of the process starting with a zero
position and a small speed, instead of a zero speed and a small position.

\begin{theorem} The following limit
$$\n'\big(F(X)\big)= \lim_{u \rightarrow 0+} u^{-\Inv{2}} \Esp_{0,u}^{\partial}\big(F(X)\big),$$
exists and defines uniquely a measure on $\E$ with $\n'({0})=0$, and
which support is included in $\E_0$.
 We have: $$\n'= \left(\frac 32\right)^\Inv 6 \Inv {\sqrt \pi} \Gamma \Big(\Inv 3\Big) \n.$$
\end{theorem}

 This formula is useful because we have more explicit
densities for the law $\Prob_{0,u}$ than for the law $\Prob_{x,0}$
(cf (\ref{mcKean joint}) and (\ref{mcKean})). For example, we can
easily infer the following corollaries:
\begin{corollary}
 The joint density of $\zeta$ and $V_{\zeta^-}$ under $\n'$ is given by:
 $$\n'(\zeta \in ds, |V_{\zeta^-}| \in dv) =  6\sqrt{\frac{2 v^3}{\pi^3 s^5}}\exp\Big(-2\frac{v^2}{s}\Big) ds dv. $$
\end{corollary}

\noindent {\bf Remark.} Taking the second marginal of this density,
this gives the $\n'$-density of $-V_{\zeta^-}$,
$$ \n'(|V_{\zeta^-}| \in dv)= \frac {45} {8 \pi} v^{-\frac 3 2} dv.$$
This improves Corollary 2 (ii) in \cite{reflecting}.

\begin{proof}
It is easy to check, for example from the corresponding property for
the free Langevin process, that $|V_{\zeta^-}| \neq 0$ $\n'$-almost
surely. But $X \rightarrow (\zeta(X), |V_{\zeta^-}|)$ is continuous
on $|V_{\zeta^-}|\neq 0$ thus we can use the limit formula to get
the density:

$$\n'(\zeta \in ds, |V_{\zeta^-}| \in dv) = \lim_{u\rightarrow 0} u^{-\Inv{2}}
\Prob_{0,u} ^\partial (\zeta \in ds, |V_{\zeta^-}| \in dv). $$

Now, using (\ref{mcKean joint}), we can calculate:
\begin{eqnarray*}
&& \frac {u^{-\Inv{2}}} {ds dv}
\Prob_{0,u} ^\partial (\zeta \in ds, |V_{\zeta^-}| \in dv) \\
&=& u^{-\Inv{2}} \frac{3v}{\pi \sqrt 2 s^2}
\exp\Big(-2\frac{u^2-vu+v^2}{s}\Big) \int_0^\frac{4 uv}{s} {\rm e}^{-\frac{3\theta}{2}} \frac{d \theta}{\sqrt{\pi\theta}} \\
&\sim& \frac{3v u^{-\Inv{2}} }{\pi \sqrt 2 s^2} \exp\Big(-2\frac{v^2}{s}\Big) \int_0^\frac{4 uv}{s} \frac{d \theta}{\sqrt{\pi\theta}}\\
&\sim& \frac{6 \sqrt 2}{\pi^\frac 3 2} \sqrt \frac {v^3}{s^5} \exp\Big(-2\frac{v^2}{s}\Big),
\end{eqnarray*}
so that we have, as stated:
 $$\n'(\zeta \in ds, |V_{\zeta^-}| \in dv) = c \sqrt \frac {v^3}{s^5} \exp\Big(-2\frac{v^2}{s}\Big) ds dv.$$
\end{proof}

\begin{corollary}
The measure $\overline \h_0(x,-u) dx du$, $x\ge 0, u\in \Real$, is
invariant for the reflected Kolmogorov process.
\end{corollary}

\begin{proof}
It is well-known that the occupation measure under the It\^o's excursion measure
$$\mu(dx,du)= \n'\left(\int_{[0,\zeta]} \un_{Z_t \in (dx,du)} dt \right),$$
is an invariant measure for the underlying Markov process (cf Theorem 8.1 in \cite{Getoor})

This enables us to calculate:
\begin{eqnarray*}
 \mu(dx,du) &=& \lim_{v \to 0} v^{-\Inv 2} \Esp_{0,v}^\partial \left(\int_{[0,\zeta]} \un_{Z_t \in (dx,du)} dt \right). \\
   &=& \lim_{v \to 0} v^{-\Inv 2} \int_{\Real_+} \Prob_{0,v}^\partial \big(Z_t \in (dx,du)\big) dt \\
   &=& dx du \lim_{v \to 0} v^{-\frac 3 2} \frac {\Prob_{x,-u} (-V_{\zeta^-} \in dv)} {dv} \qquad \hbox{ by (\ref{vPv->P(zeta)})} \\
   &=& \overline \h_0(x,-u) dx du  \qquad \qquad  \qquad \qquad \qquad \hbox{ by Lemma \ref{lemme_technique}}.
\end{eqnarray*}

\end{proof}
\begin{proof}[Proof of Theorem 2]

In order to prove $\n'= c_1 \n$, it is enough to prove that
$\n'(F(X)) = c_1 \n(F(X))$, for $F$ a Lipschitz bounded functional.
 The idea of this proof will be to compare the quantities
  $$u^{-\Inv{2}} \Esp_{0,u}^\partial \big(F(X)\big) \ \ \ \hbox{      and      }\ \ \ u^{-\Inv{2}} \Esp_{0,u}^\partial \big(F \circ \Theta_{\tau_0} (X) \big),$$
 where $\Theta$ is the usual translation operator, defined by $$\Theta_t((X_s)_{s\geq 0}) := (X_{t+s})_{s\geq 0},$$
 and $\tau_x$ is the hitting time of $x$ for the velocity process.

First we will control the difference, cutting the space on two
events, the event that $\tau_0$ is ``small", on which we will use
that $F$ is Lipschitz, and the event that $\tau_0$ is ``big", that
has a small probability. Next we will use a Markov property to see
that the quantity $u^{-\Inv{2}} \Esp_{0,u}^\partial \big(F \circ
\Theta_{\tau_0} (X) \big)$ can be compared to $\n(F)$.

As a preliminary we prove some estimates:

 $\bullet$ We write $\mathbf{P}_u$ for the law of the Brownian motion
started from $u$. We write $\tau_x$ for both the hitting time of $x$
for the velocity process under $\Prob_{0,u}^\partial$, and the
hitting time of $x$ for the Brownian motion under $\mathbf{P}_u$.
Let $a$ be a constant. A simple calculation based on the scaling
property of the Brownian motion and on the reflection principle
gives:
 \begin{eqnarray*}
 u^{-\Inv{2}} \Prob_{0,u}^\partial (\tau_0 \ge a u ) &=&  u^{-\Inv{2}} \mathbf{P}_u(\tau_0 \ge a u) \\
 &=&  u^{-\Inv{2}} \mathbf{P}_{0} (\tau_{a^{-\Inv 2} u^\Inv 2} \ge 1) \\
 &=& u^{-\Inv{2}} \mathbf{P} \big(\mathcal{N}(0,1) \in [-a^{-\Inv 2}u^\Inv 2,a^{-\Inv 2}u^\Inv 2]\big)\\
 &\leq& a^{-\Inv2} \sqrt{\frac 2 \pi} ,
\end{eqnarray*}
where $\mathcal{N}(0,1)$ is a Gaussian variable with mean zero and
variance $1$.

$\bullet$ Let us write $h$ for the supremum of the absolute value of
the velocity process. Let $b$ be a constant. We have:
\begin{eqnarray*}
  u^{-\Inv 2} \Prob_{0,u}^\partial ( h \ge b ) &\leq& u^{-\Inv 2} \Prob_{0,u}^\partial( \tau_b< \tau_0 ) +u^{-\Inv 2} \Prob_{0,u}^\partial ( h \circ \Theta_{\tau_0} \ge b) \\
   &\leq&  u^{-\Inv 2} \Prob_{0,u}( \tau_b< \tau_0 ) + u^{-\Inv 2} \int_{\Real_+} \Prob_{0,u} (Y_{\tau_0} \in dx) \Prob_{x,0}^\partial (h \ge b) \\
&\leq&  \frac {u^{\Inv 2}} {b}+  u^{-\Inv{2}} \int \Prob_{0,u}
(Y_{\tau_0} \in dx) x^\Inv{6} f(x),
\end{eqnarray*}
where the function $f : x \rightarrow x^{-\Inv{6}}
\Prob_{x,0}^\partial (h \ge b) $ is bounded and has limit $f
(0)=\n(h\ge b)$ at zero, thanks to Proposition 2. In the sum, the
second term is thus equal to:

\begin{eqnarray*}
u^{-\Inv{2}} \Esp_{0,u}\big( X_{\tau_0} ^\Inv{6} f(X_{\tau_0}) \big)
&=& u^{-\Inv{2}} \Esp_{0,1}\big( (u^3 X_{\tau_0}) ^\Inv{6} f(u^3 X_{\tau_0}) \big)\\
&=&  \Esp_{0,1}\big( X_{\tau_0} ^\Inv{6} f(u^3 X_{\tau_0}) \big)\\
& \rightarrow_{u\rightarrow 0}& \Esp_{0,1}\big( X_{\tau_0}
^\Inv{6}\big) f(0),
\end{eqnarray*}
where in the second line we used the usual scaling property for the Langevin process.

We write $c_1=\Esp_{0,1}\big( X_{\tau_0} ^\Inv{6}\big)$, so that we
have the bound:

 \begin{equation*}
     u^{-\Inv 2} \Prob_{0,u}^\partial ( h \ge b ) \leq \frac {u^{\Inv 2}} {b}+ c_1 \n(h \ge b).
 \end{equation*}

  We would like to prove that $c_1$ is finite.  We can actually calculate it explicitly.
Indeed, thanks to Lefebvre \cite{Lefebvre} we know that the density
of the variable $X_{\tau_0}$ under $\Prob_{0,1}$ is given by:
 $$\Prob_{0,1}(X_{\tau_0} \in d\xi) = \frac {\Gamma(\frac 2 3)} {3^\Inv
 6 2^{\frac 2 3} \pi} \xi^{-\frac 4 3} {\rm e} ^{-\frac 2 {9\xi}} d \xi,$$
so that we can calculate:
\begin{eqnarray*}
c_1 &=& \int_{\Real_+} \xi^\Inv 6
\Prob_{0,1}(X_{\tau_0} \in d\xi) \\
 &=& \frac {\Gamma(\frac 2 3)} {2^{\frac 2 3} 3^\Inv 6 \pi} \int_{\Real_+}
 \xi^{-\frac 7 6} {\rm e}^{-\frac 2 {9\xi}} d \xi \\
&=& \frac {\Gamma(\frac 2 3)} {2 \pi 3^\Inv 6} \int_{\Real_+} \left(\frac
9 2\right)^{\Inv 6} x^{-\frac 5 6} {\rm e}^{-x} dx \\
&=& \frac {3^\Inv 6} {2^\frac 5 6 \pi }  \Gamma\Big(\frac 2 3\Big) \Gamma\Big(\Inv 6\Big)\\
&=& \left(\frac 3 2\right)^\Inv 6 \Inv {\sqrt \pi} {\Gamma\Big(\frac 1 3\Big)},
\end{eqnarray*}
Let us notice that this is the constant that appears in the theorem.

$\bullet$ We are ready to tackle the proof of this theorem. We write $l$ for the Lipschitz constant of $F$.
We have:
\begin{eqnarray*}
u^{-\Inv{2}} \Esp_{0,u}^\partial \big( |F(e) - F \circ
\Theta_{\tau_0} (e)|\un_{\tau_0 < a u, h < b}\big)
&\leq& u^{-\Inv 2} l(a u)b \\
&\leq& a b u^\Inv 2 l,
\end{eqnarray*}
and

\begin{eqnarray*}
&& u^{-\Inv{2}} \Esp_{0,u}^\partial \big( |F(X) - F \circ
\Theta_{\tau_0} (X)| \un_{\tau_0 \ge a u \hbox{ or } h \ge b}\big) \\
&\leq& \big(2 \sup(F)\big) \left(\sqrt{\frac 2 \pi} a^{-\Inv 2} +
\frac {u^\Inv 2} b + c_1 \n(h \ge b)\right),
\end{eqnarray*}
thus we deduce
\begin{eqnarray*}
\qquad \: \limsup_{u\rightarrow 0} u^{-\Inv{2}} \Esp_{0,u}^\partial
\big( |F(X) - F \circ \Theta_{\tau_0} (X)|\big)
  &\leq& \big(2 \sup(F)\big) \left(\sqrt{\frac 2 \pi} a^{-\Inv 2} + c_1 \n(h \ge b)\right).
\end{eqnarray*}
The $\limsup$ is bounded by this expression, $a$ and $b$ being any
positive constant. Letting $a$ and $b$ go to infinity shows that
$$ \lim_{u\rightarrow 0} u^{-\Inv{2}} \Esp_{0,u}^\partial \big( |F(X) - F \circ \Theta_{\tau_0} (X)|\big) = 0.$$

$\bullet$ Next, we just need to prove that $u^{-\Inv{2}}
\Esp_{0,u}^\partial (F \circ \Theta_{\tau_0}(X))$ has a limit when
$u$ goes to zero, and that this limit is $c_1 \n(F(X))$, in order to
get that $\n'$ is well-defined and equal to $c_1 \n$.

 The calculation is similar to the one above, that we did with
$\un_{h \ge b}$ instead of $F$. Here again, the Markov property gives
us:
\begin{eqnarray*}
u^{-\Inv{2}} \Esp_{0,u}^\partial (F \circ \Theta_\tau(X)) &=&
u^{-\Inv{2}} \int \Prob_{0,u} (X_\tau \in dx) x^\Inv{6} f_F(x),
\end{eqnarray*}
where the function $f_F : x \rightarrow x^{-\Inv{6}}
\Esp_{x,0}^\partial (F(X)) $ is bounded and has limit $f_F
(0)=\n(F(X))$ at zero. We thus have:

\begin{eqnarray*}
u^{-\Inv{2}} \Esp_{0,u}^\partial (F \circ \Theta_\tau(X)) &
\rightarrow_{u\rightarrow 0}& \Esp_{0,1} \big( Y_\tau ^\Inv{6}\big)
f_F(0)= c_1 \n(F(X)),
\end{eqnarray*}
and the theorem is proved.
\end{proof}

 The second new expression we get is different, this time the measure
is given as a mixture and not as a limit. Recall that the probability measure $Q_{u;0}$ has been defined in (\ref{Q_u,0}).

 \begin{proposition} The measure $\n'$ is also given by the expression:
\begin{eqnarray}
\n'\big(F(X)\big) = \frac 3 {2\pi} \int_{\Real_+} u^{-\frac 3 2}
Q_{u;0}\big(F(\hat{X})\big)du,
\end{eqnarray}
where $(\hat{X}_t)_{0\leq t\leq \zeta}$ is defined by
$\hat{X}_t=X_{\zeta - t}$.
 \end{proposition}

 The price to pay is that we need to consider the time-reversed
excursions and to use the laws $Q_{u;0}$ instead of
$\Prob_{0,u}^\partial$. That is, the probability transitions of the
excursions are no more the ones of the Langevin process, killed at
zero, they become the $\overline \h_0$-transforms of these, as
written in (\ref{Q_u,0}).

\begin{proof}
This proposition is a consequence of the material developed in Section 4.2. Indeed, we have:

\begin{eqnarray*}
\n'(F(X))&=&\lim_{u \rightarrow 0} u^{-\Inv 2} \Esp_{0,u}^\partial (F(X)) \\
&=& \lim_{u \rightarrow 0} u^{-\frac 3 2} \int_{\Real_+} Q_{u;v}(F(X)) \varphi(u,v) dv \qquad \hbox{from (\ref{Q_u})} \\
&=& \lim_{u \rightarrow 0} \int_{\Real_+} \frac 3 {2 \pi} \frac {v^\frac 3 2}{u^3 + v^3} Q_{v;u}(F(\hat{X})) dv\\
&=& \int_{\Real_+} \frac 3 {2 \pi} v^{-\frac 3 2}
Q_{v;0}(F(\hat{X})) dv,
\end{eqnarray*}
where in the third line, we wrote the expression of $\varphi$ and used (\ref{uv-vu}).
\end{proof}

\section{Appendix}

\begin{proof}[Proof of Lemma 3]
 The first part of the lemma is just a
summary of known results, the case $x=0$ is nothing else that the
formula (\ref{mcKean}) written for the killed process (as mentioned
just after the formula), while the general case is given by
Gor$'$kov in \cite{Gorkov} and Lachal in \cite{Lachal91}. In this
article Lachal also underlines that taking $x=0$ in (\ref{h_v}) does
yield (\ref{h_v0}).

 For the second part we first prove that $\Phi_0$ and $\h$ are
well-defined and continuous \footnote{Note that
$\Phi(x,u;y,v)=\Phi_0(x-y,u;v).$}. For this we just give rough
bounds and use the theorem of dominated convergence and the theorem
of derivation under the integral. The main technical difficulty
stems from the number of variables.

  We have $$p_t(x,u;0,v) = \frac{\sqrt 3}{\pi t^2} \exp \big(-R(x,u,v,t)\big),$$
where $R(x,u,v,t)$ is the quotient:
\begin{eqnarray*}
R(x,u,v,t) &=& \frac{6}{t^3} (x+tu)^2 + \frac{6}{t^2} (x+tu)(v-u)+\frac{2}{t} (v-u)^2 \\
 &=& \Inv{t^3} \Big[\Inv{2} (3x+tu+2tv)^2 + \frac{3}{2} (x+tu)^2  \Big].
\end{eqnarray*}
The quotient $R$ is nonnegative.

 Let $(x_0,u_0,v_0)$ be in $D \times \Real$. We search for a
neighborhood of $(x_0,u_0,v_0)$ (in $D \times \Real$) on which the
integrand is bounded by an integrable function (of $t$). This will
prove that $\Phi_0$ is well-defined on this neighborhood and
continuous at $(x_0,u_0,v_0)$. We distinguish two cases:

 1) $x_0 \neq 0$: Then $R(x,u,v,t)$ is equivalent to $\frac{6
x_0^2}{t^3}$ in the neighborhood of $(x_0,u_0,v_0,0)$, thus it is
bounded below by $\frac{5 x_0^2}{t^3}$ on a $V \times ]0,\eps]$,
where $V$ is a neighborhood of $(x_0,u_0,v_0)$ and $\eps$ a strictly
positive number.

 On $V$, $p_t(x,u;0,v)$ is bounded above by the function
 $$\un_{]0,\eps]} (t) \frac{\sqrt 3}{\pi t^2} \exp\Big(-\frac{5 x^2}{t^3}\Big) + \un_{]\eps,\infty[} (t) \frac{\sqrt 3}{\pi t^2},$$
 which is integrable.

 2) $x_0=0$: Then $u_0>0$. On a neighborhood $V$ of $(0,u_0,v_0)$ we
have $u> \frac{2 u_0}{3}$, thus we have
$$R(x,u,v,t) \geq \frac{3}{2t^3}(x+tu)^2 \geq \frac{u_0^2}{t},$$
and thus the function $p_t(x,u;0,v)$ is bounded above by
$$\frac{\sqrt 3}{\pi t^2} \exp\big(-\frac{u_0}{t}\big),$$ which is
integrable.

 We thus proved the continuity of $\Phi_0$. A similar method proves
that $\Phi_0$ is infinitely differentiable. To get a continuity
result on $\h$, we will need some bounds for $\Phi_0(x,u,v)$, but
only for $v>0$.

For $v>0$, we have $R(x,u,v,t) \geq \frac{3 v^2}{2t}$, thus we have:

\begin{eqnarray*}
\Phi_0(x,u,v) &\leq& \int_0^\infty \frac{\sqrt 3}{\pi t^2} \exp\Big(- \frac{3 v^2}{2t}\Big) dt \\
&\leq& \frac{2 \sqrt 3}{3 \pi} v ^{-2}.
\end{eqnarray*}
If $(x_0,u_0,v_0)$ is a given point in $E=\Real_+^*\times D$, then in the neighborhood of this point we have $v>\frac{v_0} {2}$
and we deduce:
\begin{eqnarray*}
\frac{\mu^\frac{3}{2}}{\mu ^3 + 1} \Phi_0(x,u, \mu v) \leq \frac{8 \sqrt 3}{3 \pi} \frac{\mu^{-\frac{1}{2}} v^{-2}}{\mu ^3 + 1},
\end{eqnarray*}
which, considered as a function of $\mu$, is integrable on $\Real_+$.

 The function $\h$ is thus well-defined and continuous.

We now study the behavior of $\h$ when $v$ is small.

\begin{eqnarray*}
\Inv{v} \h_v(x,u) &=& \Phi_0(x,u,v) - \frac{3}{2 \pi} \int_0^\infty \frac{\mu^\frac{3}{2}}{\mu ^3 + 1} \Phi_0(x,u, \mu v) d\mu \\
&=& \Big[\Phi_0(x,u,0) - \frac{3}{2 \pi} \int_0^\infty \frac{\mu^\frac{3}{2}}{\mu ^3 + 1} \Phi_0(x,u, \mu v) d\mu \Big] + O(v) \\
&=& \Esp \big[\Phi_0(x,u,0) - \Phi_0(x,u,v \xi) \big] + O(v),
\end{eqnarray*}
where $\xi$ is a random variable with law the probability measure
$\displaystyle \frac{3}{2 \pi} \frac{\mu^\frac{3}{2}}{\mu ^3 + 1}
d\mu$. We next observe that:

\begin{eqnarray*}
\Esp \big[\Phi_0(x,u,0) - \Phi_0(x,u,v \xi) \big] &=& -\int_{\Real_+} \Prob(v \xi \geq \mu) \frac{\partial \Phi_0}{\partial v} (x,u,\mu) d \mu \\
&=& v^\Inv{2} \int_{\Real_+} f_v(\mu) d\mu,
\end{eqnarray*}
where we have written $f_v(\mu) = -v^{-\Inv{2}} \Prob (\xi \geq \mu
v^{-1}) \frac{\partial \Phi_0}{\partial v} (x,u,\mu)$.

 But the probability  $\Prob (\xi \geq a)$ is equivalent to $\frac 3
\pi a^{-\Inv 2}$ when $a$ goes to infinity, and bounded by the same
$\frac 3 \pi a^{-\Inv 2}$ for any $a$. On the one hand we deduce
that the continuous functions $f_v$ converge weakly to the function
$\displaystyle f_0: \mu \rightarrow -\frac 3 \pi \mu^{-\Inv 2}
\frac{\partial \Phi_0}{\partial v} (x,u,\mu)$ when $v$ goes to zero,
on the other hand that $|f_v| \leq |f_0|$. We just need to prove
that $f_0$ is integrable with respect to the Lebesgue measure. We
have:

\begin{eqnarray*}
-\frac{\partial \Phi_0}{\partial v} (x,u,v) &=& -\int_{\Real_+} \frac{\partial p_t}{\partial v} (x,u;0,v) dt \\
&=& \int_{\Real_+} \Big(\frac{6x}{t^2} + \frac{2u} t + \frac {4v} t \Big) p_t(x,u;0,v) dt \\
&=& \int_{\Real_+} \frac{2 \sqrt 3} {\pi t^4}(3x + tu + 2tv) \exp \big( - R(x,u,v,t)\big) dt.
\end{eqnarray*}

On the one hand, we have :
\begin{eqnarray*}
 && \Big| \frac{\partial \Phi_0}{\partial v} (x,u,v) \Big| \\
 &\leq& 3x \int_{\Real_+} \frac{2 \sqrt 3} {\pi t^4} \exp\Big(-\frac{3}{2t^3} (x+tu)^2\Big)
 |u + 2v| \int_{\Real_+} \frac{2 \sqrt 3} {\pi t^3} \exp \Big(-\frac{3}{2t^3} (x+tu)^2\Big) \\
&\leq& (A + B v),
\end{eqnarray*}
where $A$ and $B$ depend only on $x$ and $u$.

On the other hand, we have:
\begin{eqnarray*}
&& \Big| \frac{\partial \Phi_0}{\partial v} (x,u,v) \Big|  \\
&\leq& \frac{6 \sqrt 3 x} \pi  \int_{\Real_+} \Inv {t^4} \exp\Big(-\frac{3v^2}{2t}\Big) +
\frac{(2u+4v) \sqrt 3} \pi  \int_{\Real_+} \Inv {t^3} \exp \Big(-\frac{3v^2}{2t}\Big)  \\
 &\leq& C v^{-7} + D (u+2v) v^{-5},
\end{eqnarray*}
where $C$ and $D$ are constants.

 Let us gather the results. The function $|f_0|$ is bounded by a
$O(\mu^{-\Inv 2})$ in the neighborhood of zero and bounded by a
$O(\mu^{-3})$ in the neighborhood of infinity, thus it is
integrable.
\end{proof}

 \vspace{9pt} \noindent {\bf Acknowledgements.} I would like to thank
sincerely my thesis advisor Jean Bertoin, who led all the directions
of my work.

\bibliographystyle{abbrv}
\bibliography{AIHP322_biblio}

\end{document}